\documentclass[11pt,twoside,leqno]{article}
\usepackage[all]{xy}
\usepackage{amsfonts,amssymb,amscd,amsmath,amsthm,enumerate,verbatim,calc,latexsym}
\usepackage[colorlinks,linkcolor=black,citecolor=black,bookmarksnumbered]{hyperref}
\newtheorem{thm}{Theorem}[section]
\newtheorem{cor}[thm]{Corollary}
\newtheorem{lem}[thm]{Lemma}

\newtheorem{rem}[thm]{Remark}
\newtheorem{ques}[thm]{Question}

\numberwithin{equation}{section}
\begin{document}
\title{\bf Extension functors of generalized local cohomology modules}
\author{
\\{\bf Alireza Vahidi}\\
\small Department of Mathematics, Payame Noor University (PNU), P.O.BOX, 19395-4697, Tehran, Iran\\
\small E-mail: vahidi.ar@pnu.ac.ir\\
\\{\bf Faisal Hassani}\\
\small Department of Mathematics, Payame Noor University (PNU), P.O.BOX, 19395-4697, Tehran, Iran\\
\small E-mail: f\underline{ }hasani@pnu.ac.ir\\
\\{\bf Elham Hoseinzade}\\
\small Department of Mathematics, Payame Noor University (PNU), P.O.BOX, 19395-4697, Tehran, Iran\\
\small E-mail: el.hosseinzade.phd@gmail.com
}

\date{}
\maketitle

\renewcommand{\thefootnote}{}
\renewcommand{\thefootnote}{\arabic{footnote}}
\setcounter{footnote}{0}
\begin{abstract}
Let $R$ be a commutative Noetherian ring with non-zero identity, $\mathfrak{a}$ an ideal of $R$, $M$ a finitely generated $R$--module, and $X$ an arbitrary $R$--module. In this paper, for non-negative integers $s, t$ and a finitely generated $R$--module $N$, we study the membership of $\operatorname{Ext}^{s}_{R}(N, \operatorname{H}^{t}_{\mathfrak{a}}(M, X))$ in Serre subcategories of the category of $R$--modules and present some upper bounds for the injective dimension and the Bass numbers of $\operatorname{H}^{t}_{\mathfrak{a}}(M, X)$. We also give some results on cofiniteness and minimaxness of $\operatorname{H}^{t}_{\mathfrak{a}}(M, X)$ and finiteness of $\operatorname{Ass}_R(\operatorname{H}^{t}_{\mathfrak{a}}(M, X))$.\\


\end{abstract}
\section{Introduction}\label{1}
Throughout $R$ is a commutative Noetherian ring with non-zero identity, $\mathfrak{a}$ is an ideal of $R$, $M$ is a finite (i.e. finitely generated) $R$--module, and $X$ is an arbitrary $R$--module which is not necessarily finite. For basic results, notations, and terminology not given in this paper, readers are referred to \cite{BSh, BH, Rot}.

Let $L$ be a finite $R$--module. Grothendieck, in \cite{G}, conjectured that $\operatorname{Hom}_R(R/\mathfrak{a}, \operatorname{H}^j_\mathfrak{a}(L))$ is finite for all $j$. In \cite {Ha2}, Hartshorne gave a counter-example and raised the question whether $\operatorname{Ext}^{i}_{R}(R/\mathfrak{a}, \operatorname{H}^{j}_{\mathfrak{a}}(L))$ is finite for all $i$ and $j$. The $j$th generalized local cohomology module of $M$ and $X$ with respect to $\mathfrak{a}$,
$$\operatorname{H}^j_\mathfrak{a}(M, X)\cong \underset{n\in \mathbb{N}}\varinjlim \operatorname{Ext}^{j}_{R}(M/{\mathfrak{a}}^{n}M, X),$$
was introduced by Herzog in \cite {He}. It is clear that $\operatorname{H}^j_\mathfrak{a}(R, X)$ is just the ordinary local cohomology module $\operatorname{H}^j_\mathfrak{a}(X)$ of $X$ with respect to $\mathfrak{a}$. As a generalization of Hartshorne's question, we have the following question for generalized local cohomology modules (see \cite[Question 2.7]{Y}).

\begin{ques}
When is $\operatorname{Ext}^{i}_{R}(R/\mathfrak{a}, \operatorname{H}^{j}_{\mathfrak{a}}(M, L))$ finite for all $i$ and $j$?
\end{ques}

In this paper, we study $\operatorname{Ext}_R^i(N, \operatorname{H}_\mathfrak{a}^j(M, X))$ in general for a finite $R$--module $N$ and an arbitrary $R$--module $X$.

In Section \ref{2}, we present some technical results (Lemma \ref{2-1}, Theorem \ref{2-2}, Theorem \ref{2-9}, and Theorem \ref{2-14}) which show that, in certain situation, for non-negative integers $s$, $t$, $s'$, $t'$ with $s+ t= s'+ t'$, $\operatorname{Ext}^{s}_{R}(N, \operatorname{H}^{t}_{\mathfrak{a}}(M, X))\cong \operatorname{H}^{s'}_{\mathfrak{a}}(\operatorname{Tor}^R_{t'}(N, M), X)$ and the $R$--modules $\operatorname{Ext}_R^s(N, \operatorname{H}_\mathfrak{a}^t(M, X))$ and $\operatorname{H}^{s}_{\mathfrak{a}}(\operatorname{Tor}^{R}_{t}(N, M), X)$ are in a Serre subcategory of the category of $R$--modules (i.e. the class of $R$--modules which is closed under taking submodules, quotients, and extensions).

Section \ref{3} consists of applications. In this section, we apply the results of Section \ref{2} to some Serre subcategories (e.g. the class of zero $R$--modules and the class of finite $R$--modules) and deduce some properties of generalized local cohomology modules. In Corollaries \ref{3-1}--\ref{3-3}, we present some upper bounds for the injective dimension and the Bass numbers of generalized local cohomology modules. We study cofiniteness and minimaxness of generalized local cohomology modules in Corollaries \ref{3-4}--\ref{3-8}. Recall that, an $R$--module $X$ is said to be $\mathfrak{a}$--cofinite (resp. minimax) if $\operatorname{Supp}_R(X)\subseteq \operatorname{Var}(\mathfrak{a})$ and $\operatorname{Ext}^{i}_{R}(R/\mathfrak{a}, X)$ is finite for all $i$ \cite{Ha2} (resp. there is a finite submodule $X'$ of $X$ such that $X/X'$ is Artinian \cite{Z}) where $\operatorname{Var}(\mathfrak{a})= \{\mathfrak{p}\in \operatorname{Spec}(R)\ :\  \mathfrak{p}\supseteq \mathfrak{a}\}$. We show that if $\operatorname{Ext}^{i}_{R}(R/\mathfrak{a}, X)$ is finite for all $i\leq t$  and $\operatorname{H}^{i}_{\mathfrak{a}}(M, X)$ is minimax for all $i< t$, then $\operatorname{H}^{i}_{\mathfrak{a}}(M, X)$ is $\mathfrak{a}$--cofinite for all $i< t$ and $\operatorname{Hom}_{R}(R/\mathfrak{a}, \operatorname{H}^{t}_{\mathfrak{a}}(M, X))$ is finite. We prove that if $\operatorname{Ext}^{i}_{R}(R/\mathfrak{a}, X)$ is finite for all $i\leq \operatorname{ara}(\mathfrak{a})$, where $\operatorname{ara}(\mathfrak{a})$ is the arithmetic rank of $\mathfrak{a}$, and $\operatorname{H}^{i}_{\mathfrak{a}}(M, X)$ is $\mathfrak{a}$--cofinite for all $i\neq t$, then $\operatorname{H}^{t}_{\mathfrak{a}}(M, X)$ is also an $\mathfrak{a}$--cofinite $R$--module. We show that if $R$ is local, $\dim(R/\mathfrak{a})\leq 2$, and $\operatorname{Ext}^i_R(R/\mathfrak{a}, X)$ is finite for all $i\leq t+ 1$, then $\operatorname{H}^{i}_{\mathfrak{a}}(M, X)$ is $\mathfrak{a}$--cofinite for all $i< t$ if and only if $\operatorname{Hom}_R(R/\mathfrak{a}, \operatorname{H}^{i}_{\mathfrak{a}}(M, X))$ is finite for all $i\leq t$. We also prove that if $R$ is local, $\dim(R/\mathfrak{a})\leq 2$, $\operatorname{Ext}^i_R(R/\mathfrak{a}, X)$ is finite for all $i$, and $\operatorname{H}^{2i}_{\mathfrak{a}}(M, X)$ (or $\operatorname{H}^{2i+ 1}_{\mathfrak{a}}(M, X)$) is $\mathfrak{a}$--cofinite for all $i$, then $\operatorname{H}^{i}_{\mathfrak{a}}(M, X)$ is $\mathfrak{a}$--cofinite for all $i$. In Corollary \ref{3-9}, we state the weakest possible conditions which yield the finiteness of associated prime ideals of generalized local cohomology modules. Note that, one can apply the results of Section \ref{2} to other Serre subcategories to deduce more properties of generalized local cohomology modules.
\section{Main results}\label{2}



The following lemma is crucial in this paper.

\begin{lem}\label{2-1}
Let $X$ be an arbitrary $R$--module and let $M, N$ be finite $R$--modules. Then we have third quadrant spectral sequences
\begin{enumerate}
  \item[\emph{(i)}] ${}^{I}E_{2}^{p, q}:= \operatorname{H}^{p}_{\mathfrak{a}}(\operatorname{Tor}^{R}_{q}(N, M), X)\underset{p}\Longrightarrow H^{p+ q}$ and
  \item[\emph{(ii)}]  ${}^{II}E_{2}^{p, q}:= \operatorname{Ext}^{p}_{R}(N, \operatorname{H}^{q}_{\mathfrak{a}}(M, X))\underset{p}\Longrightarrow H^{p+ q}$.
\end{enumerate}
\end{lem}

\begin{proof}
Assume that
$$\operatorname{E}^{\bullet X}: 0\longrightarrow E^0\longrightarrow  \cdots \longrightarrow  E^i\longrightarrow \cdots$$
is a deleted injective resolution of $X$. By applying $\operatorname{Hom}_R(M, \Gamma_\mathfrak{a}(-))$ to $\operatorname{E}^{\bullet X}$, we get the complex
$$\operatorname{Hom}_R(M, \Gamma_\mathfrak{a}(\operatorname{E}^{\bullet X})): 0\longrightarrow  \operatorname{Hom}_R(M, \Gamma_\mathfrak{a}(E^0))\longrightarrow \cdots \longrightarrow \operatorname{Hom}_R(M, \Gamma_\mathfrak{a}(E^i))\longrightarrow \cdots.$$
Assume that
$$0\longrightarrow \operatorname{Hom}_R(M, \Gamma_\mathfrak{a}(\operatorname{E}^{\bullet X}))\longrightarrow T^{\bullet, 0}\longrightarrow \cdots \longrightarrow T^{\bullet, j}\longrightarrow \cdots$$
is a Cartan--Eilenberg injective resolution of $\operatorname{Hom}_R(M, \Gamma_\mathfrak{a}(\operatorname{E}^{\bullet X}))$ which exists from \cite[Theorem 10.45]{Rot}. Now, consider the third quadrant bicomplex $\mathcal{T}= \{\operatorname{Hom}_R(N, T^{p, q})\}$. We denote the total complex of $\mathcal{T}$ by $\operatorname{Tot}(\mathcal{T})$.

(i) Let ${}^{I}\operatorname{E}_{2}= \operatorname{H}'^{p}\operatorname{H}''^{p, q}(\mathcal{T})$ be the first iterated homology of $\mathcal{T}$ with respect to the first filtration. We have
$$\begin{array}{ll}
\operatorname{H}''^{p, q}(\mathcal{T})\!\!&\cong \operatorname{H}^{q}(\operatorname{Hom}_R(N, T^{p, \bullet}))\\
                       &\cong \operatorname{Ext}^{q}_R(N, \operatorname{Hom}_R(M, \Gamma_\mathfrak{a}(\operatorname{E}^{p})))\\
                       &\cong \operatorname{Hom}_R(\operatorname{Tor}^{R}_{q}(N, M), \Gamma_\mathfrak{a}(\operatorname{E}^{p}))
\end{array}$$
from \cite[Proposition 2.1.4]{BSh} and \cite[Corollary 10.63]{Rot}. Therefore, by \cite[Lemma 2.1(i)]{DST},
$$\begin{array}{ll}
{}^{I}\operatorname{E}_{2}^{p, q}\!\!&\cong \operatorname{H}'^{p}\operatorname{H}''^{p, q}(\mathcal{T})\\
                       &\cong \operatorname{H}^{p}(\operatorname{Hom}_R(\operatorname{Tor}^{R}_{q}(N, M), \Gamma_\mathfrak{a}(\operatorname{E}^{\bullet X})))\\
                       &\cong \operatorname{H}^{p}_{\mathfrak{a}}(\operatorname{Tor}^{R}_{q}(N, M), X)\\
\end{array}$$
which yields the third quadrant spectral sequence
$${}^{I}\operatorname{E}_{2}^{p, q}:= \operatorname{H}^{p}_{\mathfrak{a}}(\operatorname{Tor}^{R}_{q}(N, M), X)\underset{p}\Longrightarrow \operatorname{H}^{p+ q}(\operatorname{Tot}(\mathcal{T})).$$

(ii) Let ${}^{II}\operatorname{E}_{2}= \operatorname{H}''^{p}\operatorname{H}'^{q, p}(\mathcal{T})$ be the second iterated homology of $\mathcal{T}$ with respect to the second filtration. Note that every short exact sequence of injective modules splits and so it remains split after applying the functor $\operatorname{Hom}_R(N, -)$. By using this fact and the fact that $T^{\bullet, \bullet}$ is a Cartan--Eilenberg injective resolution of $\operatorname{Hom}_R(M, \Gamma_\mathfrak{a}(\operatorname{E}^{\bullet X}))$, we get
$$\begin{array}{ll}
\operatorname{H}'^{q, p}(\mathcal{T})\!\!&\cong \operatorname{H}^{q}(\operatorname{Hom}_R(N, T^{\bullet, p}))\\
                           &\cong \operatorname{Hom}_R(N, \operatorname{H}^{q}(T^{\bullet, p}))\\
                           &\cong \operatorname{Hom}_R(N, \operatorname{H}^{q, p}).
\end{array}$$
Therefore, by \cite[Lemma 2.1(i)]{DST},
$$\begin{array}{ll}
{}^{II}\operatorname{E}_{2}^{p, q}\!\!&\cong \operatorname{H}''^{p}\operatorname{H}'^{q, p}(\mathcal{T})\\
                        &\cong \operatorname{H}^{p}(\operatorname{Hom}_R(N, \operatorname{H}^{q, \bullet}))\\
                        &\cong \operatorname{Ext}^{p}_{R}(N, \operatorname{H}^{q}_{\mathfrak{a}}(M, X))
\end{array}$$
which gives the third quadrant spectral sequence
$${}^{II}\operatorname{E}_{2}^{p, q}:= \operatorname{Ext}^{p}_{R}(N, \operatorname{H}^{q}_{\mathfrak{a}}(M, X))\underset{p}\Longrightarrow \operatorname{H}^{p+ q}(\operatorname{Tot}(\mathcal{T}))$$
as desired.
\end{proof}


Assume that $s, t, s', t'$ are non-negative integers and that $N$ is a finite $R$--module.  In the following theorem, we provide an isomorphism between the $R$--modules $\operatorname{Ext}^{s}_{R}(N, \operatorname{H}^{t}_{\mathfrak{a}}(M, X))$ and $\operatorname{H}^{s'}_{\mathfrak{a}}(\operatorname{Tor}^R_{t'}(N, M), X)$.

\begin{thm}\label{2-2}
Suppose that $X$ is an arbitrary $R$--module, $M, N$ are finite $R$--modules, and $s, t, s', t'$ are non-negative integers such that $(n=) s+ t= s'+ t'$. Assume also that
\begin{enumerate}
  \item[\emph{(i)}] $\operatorname{Ext}^{n+ 1- i}_{R}(N, \operatorname{H}^{i}_{\mathfrak{a}}(M, X))= 0$ for all $i< t$,
  \item[\emph{(ii)}]  $\operatorname{Ext}^{n- 1- i}_{R}(N, \operatorname{H}^{i}_{\mathfrak{a}}(M, X))= 0$ for all $i> t$,
  \item[\emph{(iii)}]  $\operatorname{Ext}^{n- i}_{R}(N, \operatorname{H}^{i}_{\mathfrak{a}}(M, X))= 0$ for all $i\neq t$,
  \item[\emph{(iv)}]   $\operatorname{H}^{n+ 1- i}_{\mathfrak{a}}(\operatorname{Tor}^{R}_{i}(N, M), X)= 0$ for all $i< t'$,
  \item[\emph{(v)}]  $\operatorname{H}^{n- 1- i}_{\mathfrak{a}}(\operatorname{Tor}^{R}_{i}(N, M), X)= 0$ for all $i> t'$, and
  \item[\emph{(vi)}] $\operatorname{H}^{n- i}_{\mathfrak{a}}(\operatorname{Tor}^{R}_{i}(N, M), X)= 0$ for all $i\neq t'$.
\end{enumerate}
Then $\operatorname{Ext}^{s}_{R}(N, \operatorname{H}^{t}_{\mathfrak{a}}(M, X))\cong \operatorname{H}^{s'}_{\mathfrak{a}}(\operatorname{Tor}^R_{t'}(N, M), X)$.
\end{thm}

\begin{proof}
By Lemma \ref{2-1}(ii), there is the third quadrant spectral sequence
$${}^{II}E_{2}^{p, q}:= \operatorname{Ext}^{p}_{R}(N, \operatorname{H}^{q}_{\mathfrak{a}}(M, X))\underset{p}\Longrightarrow H^{p+ q}.$$
For all $r\geq 2$, let ${}^{II}Z_{r}^{s, t}= \operatorname{Ker}({}^{II}E_{r}^{s, t}\longrightarrow {}^{II}E_{r}^{s+r, t+1-r})$ and ${}^{II}B_{r}^{s, t}= \operatorname{Im}({}^{II}E_{r}^{s-r, t-1+r}\longrightarrow {}^{II}E_{r}^{s, t})$. Since, by assumptions (i) and (ii), ${}^{II}E_{2}^{s+ r, t+ 1- r}= 0= {{}^{II}E_{2}^{s- r, t- 1+ r}}$, ${}^{II}E_{r}^{s+r, t+1-r}= 0= {}^{II}E_{r}^{s-r, t-1+r}$. Therefore ${}^{II}Z_{r}^{s, t}= {}^{II}E_{r}^{s, t}$ and ${}^{II}B_{r}^{s, t}= 0$. Thus ${}^{II}E_{r+ 1}^{s, t}= {}^{II}E_{r}^{s ,t}$ and so
$${}^{II}E_{\infty}^{s, t}= \cdots= {}^{II}E_{n+ 2}^{s ,t}= {}^{II}E_{n+ 1}^{s ,t}= \cdots= {}^{II}E_{3}^{s ,t}= {}^{II}E_{2}^{s ,t}= \operatorname{Ext}^{s}_{R}(N, \operatorname{H}^{t}_{\mathfrak{a}}(M, X)).$$
There exists a finite filtration
$$0= \psi^{n+1}H^{n}\subseteq \psi^{n}H^{n}\subseteq \cdots \subseteq \psi^{1}H^{n}\subseteq \psi^{0}H^{n}= H^{n}$$
such that ${}^{II}E_{\infty}^{n- r, r}\cong \psi^{n- r}H^{n}/\psi^{n- r+ 1}H^{n}$ for all $r\leq n$. Note that for each $r\neq t$, ${}^{II}E^{n- r, r}_{\infty}= 0$ by assumption (iii). Hence we get
$$0= \psi^{n+ 1}H^{n}= \psi^{n}H^{n}= \cdots= \psi^{s+ 2}H^{n}= \psi^{s+ 1}H^{n}$$
and
$$\psi^{s}H^{n}= \psi^{s- 1}H^{n}= \cdots= \psi^{1}H^{n}= \psi^{0}H^{n}= H^{n}.$$
Thus ${}^{II}E^{s, t}_{\infty}\cong \psi^{s}H^{n}/\psi^{s+ 1}H^{n}= H^{n}$ and so $\operatorname{Ext}^{s}_{R}(N, \operatorname{H}^{t}_{\mathfrak{a}}(M, X))\cong H^{n}$.

\iftrue
On the other hand, by Lemma \ref{2-1}(i) and assumptions (iv), (v), and (vi), a similar argument shows that $H^{n}\cong \operatorname{H}^{s'}_{\mathfrak{a}}(\operatorname{Tor}^R_{t'}(N, M), X)$.
\fi
Therefore $\operatorname{Ext}^{s}_{R}(N, \operatorname{H}^{t}_{\mathfrak{a}}(M, X))\cong \operatorname{H}^{s'}_{\mathfrak{a}}(\operatorname{Tor}^R_{t'}(N, M), X)$ as we desired.
\end{proof}


The following corollary is an immediate application of the above theorem and generalizes \cite[Theorem 2.21]{VA}.

\begin{cor}\label{2-3}
Suppose that $X$ is an arbitrary $R$--module, $M, N$ are finite $R$--modules, and $s, t$ are non-negative integers. Assume also that
\begin{enumerate}
  \item[\emph{(i)}] $\operatorname{Ext}^{s+ t+ 1- i}_{R}(N, \operatorname{H}^{i}_{\mathfrak{a}}(M, X))= 0$ for all $i< t$,
  \item[\emph{(ii)}] $\operatorname{Ext}^{s+ t- 1- i}_{R}(N, \operatorname{H}^{i}_{\mathfrak{a}}(M, X))= 0$ for all $i> t$,
  \item[\emph{(iii)}]  $\operatorname{Ext}^{s+ t- i}_{R}(N, \operatorname{H}^{i}_{\mathfrak{a}}(M, X))= 0$ for all $i\neq t$,
  \item[\emph{(iv)}]   $\operatorname{H}^{s+ t- 1- i}_{\mathfrak{a}}(\operatorname{Tor}^{R}_{i}(N, M), X)= 0$ for all $i> 0$, and
  \item[\emph{(v)}]  $\operatorname{H}^{s+ t- i}_{\mathfrak{a}}(\operatorname{Tor}^{R}_{i}(N, M), X)= 0$ for all $i> 0$.
\end{enumerate}
Then $\operatorname{Ext}^{s}_{R}(N, \operatorname{H}^{t}_{\mathfrak{a}}(M, X))\cong \operatorname{H}^{s+ t}_{\mathfrak{a}}(N\otimes_R M, X)$.
\end{cor}

\begin{proof}
Take $s'= s+ t$ and $t'= 0$ in Theorem \ref{2-2}.
\end{proof}


The next result provides an isomorphism between the $R$--modules $\operatorname{Ext}^{s}_{R}(N, \operatorname{H}^{t}_{\mathfrak{a}}(M, X))$ and $\operatorname{Ext}^{s'}_{R}(\operatorname{Tor}^R_{t'}(N, M), X)$.

\begin{cor}\label{2-4}
Suppose that $X$ is an arbitrary $R$--module, $M, N$ are finite $R$--modules such that $\operatorname{Supp}_R(N)\cap \operatorname{Supp}_R(M)\cap \operatorname{Supp}_R(X)\subseteq \operatorname{Var}(\mathfrak{a})$, and $s, t, s', t'$ are non-negative integers such that $(n=) s+ t= s'+ t'$. Assume also that
\begin{enumerate}
  \item[\emph{(i)}] $\operatorname{Ext}^{n+ 1- i}_{R}(N, \operatorname{H}^{i}_{\mathfrak{a}}(M, X))= 0$ for all $i< t$,
  \item[\emph{(ii)}]  $\operatorname{Ext}^{n- 1- i}_{R}(N, \operatorname{H}^{i}_{\mathfrak{a}}(M, X))= 0$ for all $i> t$,
  \item[\emph{(iii)}]  $\operatorname{Ext}^{n- i}_{R}(N, \operatorname{H}^{i}_{\mathfrak{a}}(M, X))= 0$ for all $i\neq t$,
  \item[\emph{(iv)}]   $\operatorname{Ext}^{n+ 1- i}_{R}(\operatorname{Tor}^{R}_{i}(N, M), X)= 0$ for all $i< t'$,
  \item[\emph{(v)}]  $\operatorname{Ext}^{n- 1- i}_{R}(\operatorname{Tor}^{R}_{i}(N, M), X)= 0$ for all $i> t'$, and
  \item[\emph{(vi)}] $\operatorname{Ext}^{n- i}_{R}(\operatorname{Tor}^{R}_{i}(N, M), X)= 0$ for all $i\neq t'$.
\end{enumerate}
Then $\operatorname{Ext}^{s}_{R}(N, \operatorname{H}^{t}_{\mathfrak{a}}(M, X))\cong \operatorname{Ext}^{s'}_{R}(\operatorname{Tor}^R_{t'}(N, M), X)$.
\end{cor}

\begin{proof}
It follows from \cite[Lemma 2.5(c)]{VA} and Theorem \ref{2-2}.
\end{proof}


The next corollary generalizes \cite[Theorem 3.5]{ATV}.

\begin{cor}\label{2-5}
Suppose that $X$ is an arbitrary $R$--module, $M, N$ are finite $R$--modules such that $\operatorname{Supp}_R(N)\cap \operatorname{Supp}_R(M)\cap \operatorname{Supp}_R(X)\subseteq \operatorname{Var}(\mathfrak{a})$, and $s, t$ are non-negative integers. Assume also that
\begin{enumerate}
  \item[\emph{(i)}]   $\operatorname{Ext}^{s+ t+ 1- i}_{R}(N, \operatorname{H}^{i}_{\mathfrak{a}}(M, X))= 0$ for all $i< t$,
  \item[\emph{(ii)}] $\operatorname{Ext}^{s+ t- 1- i}_{R}(N, \operatorname{H}^{i}_{\mathfrak{a}}(M, X))= 0$ for all $i> t$,
  \item[\emph{(iii)}]  $\operatorname{Ext}^{s+ t- i}_{R}(N, \operatorname{H}^{i}_{\mathfrak{a}}(M, X))= 0$ for all $i\neq t$,
  \item[\emph{(iv)}]  $\operatorname{Ext}^{s+ t- 1- i}_{R}(\operatorname{Tor}^{R}_{i}(N, M), X)= 0$ for all $i> 0$, and
  \item[\emph{(v)}]  $\operatorname{Ext}^{s+ t- i}_{R}(\operatorname{Tor}^{R}_{i}(N, M), X)= 0$ for all $i> 0$.
\end{enumerate}
Then $\operatorname{Ext}^{s}_{R}(N, \operatorname{H}^{t}_{\mathfrak{a}}(M, X))\cong \operatorname{Ext}^{s+ t}_{R}(N\otimes_R M, X)$.
\end{cor}

\begin{proof}
Put $s'= s+ t$ and $t'= 0$ in Corollary \ref{2-4}.
\end{proof}


\begin{cor}\label{2-6}
\emph{(cf. \cite[Corollary 3.6]{ATV})}
Suppose that $X$ is an arbitrary $R$--module, $M, N$ are finite $R$--modules such that $\operatorname{Supp}_R(N)\cap \operatorname{Supp}_R(M)\cap \operatorname{Supp}_R(X)\subseteq \operatorname{Var}(\mathfrak{a})$, and $t$ is a non-negative integer. Assume also that
\begin{enumerate}
  \item[\emph{(i)}]   $\operatorname{Ext}^{t+ 1- i}_{R}(N, \operatorname{H}^{i}_{\mathfrak{a}}(M, X))= 0$ for all $i< t$,
  \item[\emph{(ii)}]  $\operatorname{Ext}^{t- i}_{R}(N, \operatorname{H}^{i}_{\mathfrak{a}}(M, X))= 0$ for all $i< t$,
  \item[\emph{(iii)}]  $\operatorname{Ext}^{t- 1- i}_{R}(\operatorname{Tor}^{R}_{i}(N, M), X)= 0$ for all $i> 0$, and
  \item[\emph{(iv)}]  $\operatorname{Ext}^{t- i}_{R}(\operatorname{Tor}^{R}_{i}(N, M), X)= 0$ for all $i> 0$.
\end{enumerate}
Then $\operatorname{Hom}_{R}(N, \operatorname{H}^{t}_{\mathfrak{a}}(M, X))\cong \operatorname{Ext}^{t}_{R}(N\otimes_R M, X)$.
\end{cor}

\begin{proof}
Apply Corollary \ref{2-5} with $s= 0$.
\end{proof}


Recall that a Serre subcategory $\mathcal{S}$ of the category of $R$--modules is a subclass of $R$--modules such that for any short exact sequence
\begin{equation}\label{2-7-1}
0\longrightarrow X'\longrightarrow X\longrightarrow X''\longrightarrow 0,
\end{equation}
the module $X$ is in $\mathcal{S}$ if and only if $X'$ and $X''$ are in $\mathcal{S}$. Let $\lambda: \mathcal{S}\longrightarrow \mathcal{T}$ be a function from a Serre subcategory of the category of $R$--modules $\mathcal{S}$ to a partially ordered Abelian monoid  $(\mathcal{T}, \displaystyle\circ, \preceq)$. We say that $\lambda$ is a subadditive function if $\lambda(0)= 0$ and for any short exact sequence \eqref{2-7-1} in which all the terms belong to $\mathcal{S}$, $\lambda(X')\preceq \lambda(X)$, $\lambda(X'')\preceq \lambda(X)$, and $\lambda(X)\preceq \lambda(X')\displaystyle\circ \lambda(X'')$ \cite[Definition 2.3]{VA}.

In this paper, $\mathcal{S}$ is a Serre subcategory of the category of $R$--modules, $(\mathcal{T}, \displaystyle\circ, \preceq)$ is a partially ordered Abelian monoid, and $\lambda: \mathcal{S}\longrightarrow \mathcal{T}$ is a subadditive function.

\begin{lem}\label{2-7}
\emph{(cf. \cite[Corollary 2.5]{Mel} and \cite[Proposition 3.4(i)]{HV})}
Let $N$ be a finite $R$--module, $X$ an arbitrary $R$--module, and $t$ a non-negative integer such that $\operatorname{Ext}^{i}_{R}(N, X)\in \mathcal S$ for all $i\leq t$ $($for all $i)$. Then $\operatorname{Ext}^i_{R}(L, X)\in \mathcal S$ for each finite $R$--module $L$ with $\operatorname{Supp}_R(L)\subseteq \operatorname{Supp}_R(N)$ and for all $i\leq t$ $($for all $i)$.
\end{lem}

\begin{proof}
\iftrue
Use an induction on $t$. Note that, by Gruson's theorem \cite[Theorem 4.1]{Vas} and for a finite $R$--module $L$, there is a finite filtration
$$0= L_0\subset L_1\subset \cdots\subset L_{n-1}\subset L_n= L$$
of submodules of $L$ such that for all $1\leq j\leq n$, there exists a short exact sequence
$$0\longrightarrow L'_j\longrightarrow N^{\alpha_j}\longrightarrow L_j/L_{j- 1}\longrightarrow 0,$$
where $\alpha_j$ is an integer.
\fi
\end{proof}


In \cite[Theorem 2.10(b)]{VA}, it is shown that $\operatorname{Hom}_{R}(R/\mathfrak{a}, \operatorname{H}^{t}_{\mathfrak{a}}(M, X))\cong \operatorname{Ext}^{t}_{R}(M/\mathfrak{a}M, X)$ if $\operatorname{H}^{i}_{\mathfrak{a}}(X)= 0$ for all $i< t$. We improve this result in the following corollary. Note that, by \cite[Proposition 2.3]{HV} or \cite[Corollary 2.9]{VA}, $\operatorname{H}^{i}_{\mathfrak{a}}(M, X)= 0$ for all $i< t$ when $\operatorname{H}^{i}_{\mathfrak{a}}(X)= 0$ for all $i< t$.

\begin{cor}\label{2-8}
Let $X$ be an arbitrary $R$--module, $M$ a finite $R$--module, and $t$ a non-negative integer such that $\operatorname{H}^{i}_{\mathfrak{a}}(M, X)= 0$ for all $i< t$. Then, for every finite $\mathfrak{a}$-torsion $R$--module $N$, we have $\operatorname{Hom}_{R}(N, \operatorname{H}^{t}_{\mathfrak{a}}(M, X))\cong \operatorname{Ext}^{t}_{R}(N\otimes_R M, X)$. In particular, $\operatorname{Hom}_{R}(R/\mathfrak{a}, \operatorname{H}^{t}_{\mathfrak{a}}(M, X))\cong \operatorname{Ext}^{t}_{R}(M/\mathfrak{a}M, X)$ and so $\operatorname{Ass}_R(\operatorname{H}^{t}_{\mathfrak{a}}(M, X))= \operatorname{Ass}_R(\operatorname{Ext}^{t}_{R}(M/\mathfrak{a}M, X))$.
\end{cor}

\begin{proof}
By \cite[Corollary 2.17]{VA}, we have $\operatorname{Ext}^{i}_{R}(M/\mathfrak{a}M, X)= 0$ for all $i< t$. Therefore, from Lemma \ref{2-7}, $\operatorname{Ext}^{t- 1- i}_{R}(\operatorname{Tor}^{R}_{i}(N, M), X)= 0= \operatorname{Ext}^{t- i}_{R}(\operatorname{Tor}^{R}_{i}(N, M), X)$ for all $i> 0$. Thus $\operatorname{Hom}_{R}(N, \operatorname{H}^{t}_{\mathfrak{a}}(M, X))\cong \operatorname{Ext}^{t}_{R}(N\otimes_R M, X)$ by Corollary \ref{2-6}. For the last part, note that we have
$$\operatorname{Ass}_R(\operatorname{Hom}_{R}(R/\mathfrak{a}, \operatorname{H}^{t}_{\mathfrak{a}}(M, X)))= \operatorname{Supp}_R(R/\mathfrak{a})\cap \operatorname{Ass}_{R}(\operatorname{H}^{t}_{\mathfrak{a}}(M, X))= \operatorname{Ass}_{R}(\operatorname{H}^{t}_{\mathfrak{a}}(M, X))$$
from \cite[Exercise 1.2.27]{BH}.
\end{proof}


In the next theorem which is a generalization of \cite[Theorem 2.13]{VA}, we present sufficient conditions which convince us that $\operatorname{Ext}^{s}_{R}(N, \operatorname{H}^{t}_{\mathfrak{a}}(M, X))$ is in a Serre subcategory of the category of $R$--modules. 

\begin{thm}\label{2-9}
Suppose that $X$ is an arbitrary $R$--module, $M, N$ are finite $R$--modules, and $s, t$ are non-negative integers. Assume also that
\begin{enumerate}
  \item[\emph{(i)}]  $\operatorname{H}^{s+ t- i}_{\mathfrak{a}}(\operatorname{Tor}^{R}_{i}(N, M), X)\in \mathcal S$ for all $i$,
  \item[\emph{(ii)}]   $\operatorname{Ext}^{s+ t+ 1- i}_{R}(N, \operatorname{H}^{i}_{\mathfrak{a}}(M, X))\in \mathcal S$ for all $i< t$, and
  \item[\emph{(iii)}] $\operatorname{Ext}^{s+ t- 1- i}_{R}(N, \operatorname{H}^{i}_{\mathfrak{a}}(M, X))\in \mathcal S$ for all $i> t$.
\end{enumerate}
Then $\operatorname{Ext}^{s}_{R}(N, \operatorname{H}^{t}_{\mathfrak{a}}(M, X)) \in \mathcal S$ and
\begin{flushleft}
$\quad\lambda(\operatorname{Ext}^{s}_{R}(N, \operatorname{H}^{t}_{\mathfrak{a}}(M, X)))\preceq(\underset{i= 0}{\overset{s+ t}\circ}\lambda(\operatorname{H}^{s+ t- i}_{\mathfrak{a}}(\operatorname{Tor}^{R}_{i}(N, M), X)))$
\end{flushleft}
\begin{flushright}
$\displaystyle\circ (\underset{i= 0}{\overset{t- 1}\circ}\lambda(\operatorname{Ext}^{s+ t+ 1- i}_{R}(N, \operatorname{H}^{i}_{\mathfrak{a}}(M, X))))\displaystyle\circ (\underset{i= t+ 1}{\overset{s+ t- 1}\circ}\lambda(\operatorname{Ext}^{s+ t- 1- i}_{R}(N, \operatorname{H}^{i}_{\mathfrak{a}}(M, X)))).\quad$
\end{flushright}
\end{thm}

\begin{proof}
By Lemma \ref{2-1}(i), there is the third quadrant spectral sequence
$${}^{I}E_{2}^{p, q}:= \operatorname{H}^{p}_{\mathfrak{a}}(\operatorname{Tor}^{R}_{q}(N, M), X)\underset{p}\Longrightarrow H^{p+ q}.$$
For all $i\leq s+ t$, we have ${}^{I}E_{\infty}^{s+ t- i, i}= {}^{I}E_{s+ t+ 2}^{s+ t- i, i}$ because ${}^{I}E_{j}^{s+ t- i- j, i- 1+ j}= 0= {}^{I}E_{j}^{s+ t- i+ j, i+ 1- j}$ for all $j\geq s+ t+ 2$; so that ${}^{I}E_{\infty}^{s+ t- i, i}$ is in $\mathcal{S}$ and $\lambda({}^{I}E_{\infty}^{s+ t- i, i})\preceq \lambda({}^{I}E_{2}^{s+ t- i, i})$ from the fact that ${}^{I}E_{s+ t+ 2}^{s+ t- i, i}$ is a subquotient of ${}^{I}E_{2}^{s+ t- i, i}$ which is in $\mathcal{S}$ by assumption (i). There exists a finite filtration
$$0= \phi^{s+ t+ 1}H^{s+ t}\subseteq \phi^{s+ t}H^{s+ t}\subseteq \cdots \subseteq \phi^{1}H^{s+ t}\subseteq \phi^{0}H^{s+ t}= H^{s+ t}$$
such that ${}^{I}E_{\infty}^{s+ t- i, i}\cong \phi^{s+ t- i}H^{s+ t}/\phi^{s+ t- i+ 1}H^{s+ t}$ for all $i\leq s+ t$. Now the exact sequences
$$0\longrightarrow \phi^{s+ t- i+ 1}H^{s+ t}\longrightarrow \phi^{s+ t- i}H^{s+ t}\longrightarrow {}^{I}E_{\infty}^{s+ t- i, i}\longrightarrow 0,$$
for all $i\leq s+ t$, show that $H^{s+ t}$ is in $\mathcal{S}$ and
$$\lambda(H^{s+ t})\preceq \underset{i= 0}{\overset{s+ t}\circ}\lambda({}^{I}E_{\infty}^{s+ t- i, i})\preceq \underset{i= 0}{\overset{s+ t}\circ}\lambda({}^{I}E_{2}^{s+ t- i, i}).$$

On the other hand, by Lemma \ref{2-1}(ii), there is the third quadrant spectral sequence
$${}^{II}E_{2}^{p, q}:= \operatorname{Ext}^{p}_{R}(N, \operatorname{H}^{q}_{\mathfrak{a}}(M, X))\underset{p}\Longrightarrow H^{p+ q}.$$
Thus there exists a finite filtration
$$0= \psi^{s+t+1}H^{s+t}\subseteq \psi^{s+t}H^{s+t}\subseteq \cdots \subseteq \psi^{1}H^{s+t}\subseteq \psi^{0}H^{s+t}= H^{s+ t}$$
such that ${}^{II}E_{\infty}^{s+ t- i, i}\cong \psi^{s+ t- i}H^{s+ t}/\psi^{s+ t- i+ 1}H^{s+ t}$ for all $i\leq s+ t$. Since $H^{s+ t}$ is in $\mathcal{S}$, $\psi^{s}H^{s+ t}$ is in $\mathcal{S}$. Hence ${}^{II}E_{\infty}^{s, t}\cong \psi^{s}H^{s+ t}/\psi^{s+ 1}H^{s+ t}$ is in $\mathcal{S}$ and $\lambda({}^{II}E_{\infty}^{s, t})\preceq \lambda(\psi^{s}H^{s+ t})\preceq \lambda(H^{s+ t}).$ Therefore ${}^{II}E_{s+ t+ 2}^{s, t}$ is in $\mathcal{S}$ and
$$\lambda({}^{II}E_{s+ t+ 2}^{s, t})\preceq \lambda(H^{s+ t})$$
because ${}^{II}E_{s+ t+ 2}^{s, t}= {}^{II}E_{\infty}^{s, t}$ from the fact that ${}^{II}E_{j}^{s- j, t- 1+ j}= 0= {}^{II}E_{j}^{s+ j, t+ 1- j}$ for all $j\geq s+ t+ 2$. For all $r\geq 2$, let ${}^{II}Z_{r}^{s, t}= \operatorname{Ker}({}^{II}E_{r}^{s, t}\longrightarrow {}^{II}E_{r}^{s+r, t+1-r})$ and ${}^{II}B_{r}^{s, t}= \operatorname{Im}({}^{II}E_{r}^{s-r, t- 1+ r}\longrightarrow {}^{II}E_{r}^{s, t})$. We have the exact sequences
$$0\longrightarrow {}^{II}Z_{r}^{s, t}\longrightarrow {}^{II}E_{r}^{s, t}\longrightarrow {}^{II}E_{r}^{s, t}/{}^{II}Z_{r}^{s, t}\longrightarrow 0$$
and
$$0\longrightarrow {}^{II}B_{r}^{s ,t}\longrightarrow {}^{II}Z_{r}^{s, t}\longrightarrow {}^{II}E_{r+1}^{s, t}\longrightarrow 0.$$
Since ${}^{II}E_{2}^{s+ r, t+ 1- r}$ and ${}^{II}E_{2}^{s- r, t- 1+ r}$ are in $\mathcal{S}$ by assumptions (ii) and (iii), ${}^{II}E_{r}^{s+r, t+1-r}$ and ${}^{II}E_{r}^{s-r, t- 1+ r}$ are also in $\mathcal{S}$, and so ${}^{II}E_{r}^{s, t}/{}^{II}Z_{r}^{s, t}$ and ${}^{II}B_{r}^{s, t}$ are in $\mathcal{S}$. It shows that ${}^{II}E_{r}^{s, t}$ is in $\mathcal{S}$ whenever ${}^{II}E_{r+1}^{s ,t}$ is in $\mathcal{S}$ and we get
$$\begin{array}{ll}
\lambda({}^{II}E_{r}^{s, t})\!\!&\preceq \lambda({}^{II}E_{r+1}^{s, t})\displaystyle\circ \lambda({}^{II}E_{r}^{s, t}/{}^{II}Z_{r}^{s, t})\displaystyle\circ\lambda({}^{II}B_{r}^{s ,t})\\
&\preceq \lambda({}^{II}E_{r+1}^{s, t})\displaystyle\circ \lambda({}^{II}E_{r}^{s+r, t+1-r})\displaystyle\circ \lambda({}^{II}E_{r}^{s-r, t- 1+ r})\\
&\preceq \lambda({}^{II}E_{r+1}^{s, t})\displaystyle\circ \lambda({}^{II}E_{2}^{s+r, t+1-r})\displaystyle\circ \lambda({}^{II}E_{2}^{s-r, t- 1+ r}).
\end{array}$$
Therefore ${}^{II}E_{2}^{s, t}$ is in $\mathcal{S}$ and we have
$$\begin{array}{ll}
\lambda({}^{II}E_{2}^{s, t})\!\!\!
&\preceq \lambda({}^{II}E_{3}^{s, t})\displaystyle\circ \lambda({}^{II}E_{2}^{s+ 2, t- 1})\displaystyle\circ  \lambda({}^{II}E_{2}^{s- 2, t+ 1})\\
&\preceq \lambda({}^{II}E_{4}^{s, t})\displaystyle\circ \lambda({}^{II}E_{2}^{s+ 3, t- 2})\displaystyle\circ \lambda({}^{II}E_{2}^{s+ 2, t- 1})\displaystyle\circ \lambda({}^{II}E_{2}^{s- 2, t+ 1})\displaystyle\circ \lambda({}^{II}E_{2}^{s- 3, t+ 2})\\
&\preceq \cdots\\
&\preceq \lambda({}^{II}E_{s+ t+ 2}^{s, t})\circ (\underset{i= 0}{\overset{t- 1}\circ}\lambda({}^{II}E_2^{s+ t+ 1- i, i}))\circ  (\underset{i= t+ 1}{\overset{s+ t- 1}\circ}\lambda({}^{II}E_2^{s+ t- 1- i, i}))
\end{array}$$
which completes the proof.
\end{proof}


As an immediate application of the above theorem, we generalize \cite[Theorem 2.3]{ATV} in the next result.

\begin{cor}\label{2-10}
Suppose that $X$ is an arbitrary $R$--module, $M, N$ are finite $R$--modules such that $\operatorname{Supp}_R(N)\cap \operatorname{Supp}_R(M)\cap \operatorname{Supp}_R(X)\subseteq \operatorname{Var}(\mathfrak{a})$, and $s, t$ are non-negative integers. Assume also that
\begin{enumerate}
  \item[\emph{(i)}]  $\operatorname{Ext}^{s+ t- i}_{R}(\operatorname{Tor}^{R}_{i}(N, M), X)\in \mathcal S$ for all  $i$,
  \item[\emph{(ii)}]   $\operatorname{Ext}^{s+ t+ 1- i}_{R}(N, \operatorname{H}^{i}_{\mathfrak{a}}(M, X))\in \mathcal S$ for all $i< t$, and
  \item[\emph{(iii)}] $\operatorname{Ext}^{s+ t- 1- i}_{R}(N, \operatorname{H}^{i}_{\mathfrak{a}}(M, X))\in \mathcal S$ for all $i> t$.
\end{enumerate}
Then $\operatorname{Ext}^{s}_{R}(N, \operatorname{H}^{t}_{\mathfrak{a}}(M, X)) \in \mathcal S$ and
\begin{flushleft}
$\quad\lambda(\operatorname{Ext}^{s}_{R}(N, \operatorname{H}^{t}_{\mathfrak{a}}(M, X)))\preceq(\underset{i= 0}{\overset{s+ t}\circ}\lambda(\operatorname{Ext}^{s+ t- i}_{R}(\operatorname{Tor}^{R}_{i}(N, M), X)))$
\end{flushleft}
\begin{flushright}
$\displaystyle\circ (\underset{i= 0}{\overset{t- 1}\circ}\lambda(\operatorname{Ext}^{s+ t+ 1- i}_{R}(N, \operatorname{H}^{i}_{\mathfrak{a}}(M, X))))\displaystyle\circ (\underset{i= t+ 1}{\overset{s+ t- 1}\circ}\lambda(\operatorname{Ext}^{s+ t- 1- i}_{R}(N, \operatorname{H}^{i}_{\mathfrak{a}}(M, X)))).\quad$
\end{flushright}
\end{cor}

\begin{proof}
This follows from \cite[Lemma 2.5(c)]{VA} and Theorem \ref{2-9}.
\end{proof}


Recall that, a Serre subcategory of the category of $R$--modules $\mathcal{S}$ is said to be Melkersson with respect to $\mathfrak{a}$ if for any $\mathfrak{a}$--torsion $R$--module $X$, $(0:_{X}\mathfrak{a})$ is in $\mathcal{S}$ implies that $X$ is in $\mathcal{S}$ (see \cite [Definition 2.1]{AM} and \cite [Definition 3.1]{ATV}).

\begin{cor}\label{2-11}
\emph{(cf. \cite[Corollary 2.4]{ATV} and \cite[Theorem 2.3]{VHH})}
Suppose that $X$ is an arbitrary $R$--module, $M, N$ are finite $R$--modules such that $\operatorname{Supp}_R(N)\cap \operatorname{Supp}_R(M)\cap \operatorname{Supp}_R(X)\subseteq \operatorname{Var}(\mathfrak{a})$, and $t$ is a non-negative integer. Assume also that
\begin{enumerate}
  \item[\emph{(i)}]  $\operatorname{Ext}^{t- i}_{R}(\operatorname{Tor}^{R}_{i}(N, M), X)$ is in $\mathcal{S}$ for all $i\leq t$, and
  \item[\emph{(ii)}]   $\operatorname{Ext}^{t+ 1- i}_{R}(N, \operatorname{H}^{i}_{\mathfrak{a}}(M, X))$ is in $\mathcal{S}$ for all $i< t$.
\end{enumerate}
Then $\operatorname{Hom}_{R}(N, \operatorname{H}^{t}_{\mathfrak{a}}(M, X))\in \mathcal{S}$. In particular, $\operatorname{H}^{t}_{\mathfrak{a}}(M, X)\in \mathcal{S}$ whenever $N= R/\mathfrak{a}$ and $\mathcal{S}$ is Melkersson with respect to $\mathfrak{a}$.
\end{cor}

\begin{proof}
Take $s= 0$ in Corollary \ref{2-10}.
\end{proof}


\begin{rem}\label{2-12}
Let $X$ be an arbitrary $R$--module, $M$ a finite $R$--module, and $t$ a non-negative integer such that $\operatorname{Ext}^{t}_{R}(M/\mathfrak{a}M, X)\in \mathcal S$ and $\operatorname{H}^{j}_{\mathfrak{a}}(X) \in \mathcal S$ for all $j< t$. Let $i$ be an integer such that $1\leq i\leq t$. Since $\operatorname{Ext}^{t- i- j}_{R}(\operatorname{Tor}^{R}_{i}(R/\mathfrak{a}, M), \operatorname{H}^{j}_{\mathfrak{a}}(X))\in \mathcal S$ for all $j\leq t- i$, $\operatorname{Ext}^{t- i}_{R}(\operatorname{Tor}^{R}_{i}(R/\mathfrak{a}, M), X)\in \mathcal S$ from \emph{\cite[Theorem 2.1]{ATV}}. Thus $\operatorname{Ext}^{t- i}_{R}(\operatorname{Tor}^{R}_{i}(R/\mathfrak{a}, M), X)\in \mathcal S$ for all $i\leq t$. On the other hand, we have $\operatorname{H}^{i}_{\mathfrak{a}}(M, X)\in \mathcal S$ for all $i<t$ by \emph{\cite[Corollary 2.9]{VA}}. Hence $\operatorname{Ext}^{t+ 1- i}_{R}(R/\mathfrak{a}, \operatorname{H}^{i}_{\mathfrak{a}}(M, X))\in \mathcal S$ for all $i< t$. Therefore $\operatorname{Hom}_{R}(R/\mathfrak{a}, \operatorname{H}^{t}_{\mathfrak{a}}(M, X)) \in \mathcal S$ from \emph{Corollary \ref{2-11}}. Thus \emph{Corollary \ref{2-11}} improves \emph{\cite[Theorem 2.10(a) and Corollary 2.12]{VA}}.
\end{rem}


\begin{cor}\label{2-13}
\emph{(cf. \cite[Theorem 4.4]{ATV} and \cite[Theorem 2.7]{VHH})}
Suppose that $X$ is an arbitrary $R$--module, $M, N$ are finite $R$--modules such that $\operatorname{Supp}_R(N)\cap \operatorname{Supp}_R(M)\cap \operatorname{Supp}_R(X)\subseteq \operatorname{Var}(\mathfrak{a})$, and $t$ is a non-negative integer. Assume also that
\begin{enumerate}
  \item[\emph{(i)}]  $\operatorname{Ext}^{t+ 1- i}_{R}(\operatorname{Tor}^{R}_{i}(N, M), X)\in \mathcal S$ for all $i\leq t+ 1$, and
  \item[\emph{(ii)}]  $\operatorname{Ext}^{t+ 2- i}_{R}(N, \operatorname{H}^{i}_{\mathfrak{a}}(M, X))\in \mathcal S$ for all $i< t$.
\end{enumerate}
Then $\operatorname{Ext}^1_{R}(N, \operatorname{H}^{t}_{\mathfrak{a}}(M, X)) \in \mathcal S$.
\end{cor}

\begin{proof}
Put $s= 1$ in Corollary \ref{2-10}.
\end{proof}


\begin{thm}\label{2-14}
Suppose that $X$ is an arbitrary $R$--module, $M, N$ are finite $R$--modules, and $s, t$ are non-negative integers. Assume also that
\begin{enumerate}
  \item[\emph{(i)}]  $\operatorname{Ext}^{s+ t- i}_{R}(N, \operatorname{H}^{i}_{\mathfrak{a}}(M, X))\in \mathcal S$ for all $i$,
  \item[\emph{(ii)}]   $\operatorname{H}^{s+ t+ 1- i}_{\mathfrak{a}}(\operatorname{Tor}^{R}_{i}(N, M), X)\in \mathcal S$ for all $i< t$, and
  \item[\emph{(iii)}] $\operatorname{H}^{s+ t- 1- i}_{\mathfrak{a}}(\operatorname{Tor}^{R}_{i}(N, M), X)\in \mathcal S$ for all $i> t$.
\end{enumerate}
Then $\operatorname{H}^{s}_{\mathfrak{a}}(\operatorname{Tor}^{R}_{t}(N, M), X)\in \mathcal S$ and
\begin{flushleft}
$\quad\lambda(\operatorname{H}^{s}_{\mathfrak{a}}(\operatorname{Tor}^{R}_{t}(N, M), X))\preceq(\underset{i= 0}{\overset{s+ t}\circ}\lambda(\operatorname{Ext}^{s+ t- i}_{R}(N, \operatorname{H}^{i}_{\mathfrak{a}}(M, X))))$
\end{flushleft}
\begin{flushright}
$\displaystyle\circ (\underset{i= 0}{\overset{t- 1}\circ}\lambda(\operatorname{H}^{s+ t+ 1- i}_{\mathfrak{a}}(\operatorname{Tor}^{R}_{i}(N, M), X)))\displaystyle\circ (\underset{i= t+ 1}{\overset{s+ t- 1}\circ}\lambda(\operatorname{H}^{s+ t- 1- i}_{\mathfrak{a}}(\operatorname{Tor}^{R}_{i}(N, M), X))).\quad$
\end{flushright}
\end{thm}

\begin{proof}
The proof is sufficiently similar to that of Theorem \ref{2-9} to be omitted. We leave the proof to the reader.
\end{proof}


\begin{cor}\label{2-15}
\emph{(cf. \cite[Theorem 2.13]{VHH})}
Suppose that $X$ is an arbitrary $R$--module, $M, N$ are finite $R$--modules, and $s$ is a non-negative integer. Assume also that
\begin{enumerate}
  \item[\emph{(i)}]  $\operatorname{Ext}^{s- i}_{R}(N, \operatorname{H}^{i}_{\mathfrak{a}}(M, X))\in \mathcal S$ for all $i\leq s$, and
  \item[\emph{(ii)}] $\operatorname{H}^{s- 1- i}_{\mathfrak{a}}(\operatorname{Tor}^{R}_{i}(N, M), X)\in \mathcal S$ for all $0< i< s$.
\end{enumerate}
Then $\operatorname{H}^s_{\mathfrak{a}}(N\otimes_RM, X)\in \mathcal S$. In particular, if $\operatorname{Supp}_R(N)\cap \operatorname{Supp}_R(M)\cap \operatorname{Supp}_R(X)\subseteq \operatorname{Var}(\mathfrak{a})$ $($resp. $N= R/\mathfrak{a})$, then $\operatorname{Ext}^s_{R}(N\otimes_RM, X)\in \mathcal S$ $($resp. $\operatorname{Ext}^s_{R}(M/\mathfrak{a}M, X)\in \mathcal S)$.
\end{cor}

\begin{proof}
Put $t= 0$ in Theorem \ref{2-14}.
The last part follows from \cite[Lemma 2.5(c)]{VA}.
\end{proof}


Recall that, an $R$--module $X$ is said to be $(\mathcal S, \mathfrak{a})$--cofinite if $\operatorname{Supp}_R(X)\subseteq \operatorname{Var}(\mathfrak{a})$ and $\operatorname{Ext}^{i}_{R}(R/\mathfrak{a}, X)\in \mathcal S$ for all $i$ \cite[Definition 4.1]{ATV}. Note that when $\mathcal S$ is the category of finite $R$--modules, X is $(\mathcal S, \mathfrak{a})$--cofinite exactly when X is $\mathfrak{a}$--cofinite.

\begin{cor}\label{2-16}
\emph{(cf. \cite[Corollary 3.10]{Mel}, \cite[Theorem 4.2]{ATV}, and \cite[Corollary 2.14]{VHH})} Let $M$ be a finite $R$--module, $X$ an arbitrary $R$--module, and $t$ a non-negative integer such that $\operatorname{H}^{i}_{\mathfrak{a}}(M, X)$ is $(\mathcal S, \mathfrak{a})$--cofinite for all $i\leq t$ $($for all $i)$. Then $\operatorname{Ext}^i_{R}(M/\mathfrak{a}M, X)$ is in $\mathcal S$ for all $i\leq t$ $($for all $i)$.
\end{cor}

\begin{proof}
Follows from Lemma \ref{2-7}, \cite[Lemma 2.5(c)]{VA}, Corollary \ref{2-15}, and using an induction on $t$.
\end{proof}
\section{More applications}\label{3}



Let $X$ be an arbitrary $R$--module. We denote $\operatorname{id}_R(X)$ and $\operatorname{fd}_R(X)$ as the injective dimension and the flat dimension of $X$, respectively. In \cite{V} and for a given non-negative integer $t$, the first author compared the injective dimension of local cohomology module $\operatorname{H}^{t}_\mathfrak{a}(X)$ with the injective dimension of $X$ and those of some other local cohomology modules $\operatorname{H}^{i}_\mathfrak{a}(X)$ where $i\neq t$. He proved, in \cite[Theorem 2.8]{V}, that the inequality
$$\operatorname{id}_R(\operatorname{H}^{t}_\mathfrak{a}(X))\leq \sup \ \{\operatorname{id}_{R}(X)- t\}\cup \{\operatorname{id}_{R}(\operatorname{H}^{i}_\mathfrak{a}(X))- t+ i- 1: i< t\} \cup \{\operatorname{id}_{R}(\operatorname{H}^{i}_\mathfrak{a}(X))- t+ i+ 1: i> t\}$$
holds true. In the next result, we generalize the above inequality.

\begin{cor}\label{3-1}
Let $M$ be a finite $R$--module, $X$ an arbitrary $R$--module, and $t$ a non-negative integer. Then
\begin{flushleft}
$\quad\operatorname{id}_R(\operatorname{H}^{t}_\mathfrak{a}(M, X))\leq\sup \ \{\operatorname{fd}_R(M)+ \operatorname{id}_{R}(X)- t\}$
\end{flushleft}
\begin{flushright}
$\cup \{\operatorname{id}_{R}(\operatorname{H}^{i}_\mathfrak{a}(M, X))- t+ i- 1: i< t\}\cup \{\operatorname{id}_{R}(\operatorname{H}^{i}_\mathfrak{a}(M, X))- t+ i+ 1: i> t\}.\quad$
\end{flushright}
\end{cor}

\begin{proof}
Assume that the right-hand side of the inequality is a finite number, say $r$. Assume also that $s> r$ and $N$ is a finite $R$--module. Apply Theorem \ref{2-9} with $\mathcal{S}=$ the category of zero $R$--modules to get $\operatorname{Ext}^{s}_{R}(N, \operatorname{H}^{t}_{\mathfrak{a}}(M, X))= 0$. The result follows.
\end{proof}


For a prime ideal $\mathfrak{p}$ of $R$, the number $\mu^i(\mathfrak{p}, X)= \operatorname{dim}_{\kappa(\mathfrak{p})}(\operatorname{Ext}^i_{R_\mathfrak{p}}(\kappa(\mathfrak{p}), X_\mathfrak{p}))$ is known as the $i$th Bass number of $X$ with respect to $\mathfrak{p}$, where $\kappa(\mathfrak{p})= R_\mathfrak{p}/\mathfrak{p} R_\mathfrak{p}$. When $R$ is local with maximal ideal $\mathfrak{m}$, we write $\mu^i(X)= \mu^i(\mathfrak{m}, X)$  and $\kappa= R/\mathfrak{m}$. For given non-negative integers $s$ and $t$, it was shown in \cite[Theorem 2.6]{DY2} that
$$\mu^{s}(\operatorname{H}^{t}_\mathfrak{a}(X))\leq \mu^{s+ t}(X)+ \displaystyle\sum_{i=0}^{t- 1}\mu^{s+ t+ 1- i}(\operatorname{H}^{i}_\mathfrak{a}(X))+  \displaystyle\sum_{i=t+ 1}^{s+ t- 1}\mu^{s+ t- 1- i}(\operatorname{H}^{i}_\mathfrak{a}(X)).$$
We generalize the above inequality in the following corollary.

\begin{cor}\label{3-2}
Let $R$ be a local ring, $M$ a finite $R$--module, $X$ an arbitrary $R$--module, and $s, t$ non-negative integers. Then
\begin{flushleft}
$\quad\quad\quad\mu^{s}(\operatorname{H}^{t}_\mathfrak{a}(M, X))\leq
\displaystyle\sum_{i= 0}^{s+ t} \operatorname{dim}_\kappa(\operatorname{Ext}^{s+ t- i}_{R}(\operatorname{Tor}^{R}_{i}(\kappa, M), X))$
\end{flushleft}
\begin{flushright}
$+ \displaystyle\sum_{i=0}^{t- 1}\mu^{s+ t+ 1- i}(\operatorname{H}^{i}_\mathfrak{a}(M, X))+  \displaystyle\sum_{i=t+ 1}^{s+ t- 1}\mu^{s+ t- 1 -i}(\operatorname{H}^{i}_\mathfrak{a}(M, X)).\quad\quad\quad$
\end{flushright}
\end{cor}

\begin{proof}
Take $N= \kappa$ in Corollary \ref{2-10}.
\end{proof}


As an application of the above corollary, we have the following.

\begin{cor}\label{3-3}
Let $R$ be a local Cohen-Macaulay ring with maximal ideal $\mathfrak{m}$ and let $M$ be a finite $R$--module with finite projective dimension. Then
$$\mu^{s}(\operatorname{H}^{\operatorname{dim}(R)}_\mathfrak{m}(M, R))\leq \operatorname{dim}_\kappa(\operatorname{Hom}_{R}(\operatorname{Tor}^{R}_{s+ \operatorname{dim}(R)}(\kappa, M), R))$$
for all $s$.
\end{cor}

\begin{proof}
By \cite[Corollary 2.9]{VA} and \cite[Corollary 3.2]{CH}, $\operatorname{H}^{i}_\mathfrak{m}(M, R)= 0$ for all $i\neq \operatorname{dim}(R)$. Thus the assertion follows from Corollary \ref{3-2}.
\end{proof}


In the following two results, we generalize \cite[Proposition 4.10 and Corollary 4.11]{ATV}. Note that, by Corollary \ref{2-16}, $\operatorname{Ext}^{i}_{R}(R/\mathfrak{a}, X)$ is finite for all $i< t$ whenever $\operatorname{H}^{i}_{\mathfrak{a}}(X)$ is $\mathfrak{a}$--cofinite for all $i< t$.

\begin{cor}\label{3-4}
Let $M$ be a finite $R$--module, $X$ an arbitrary $R$--module, and $s, t$ non-negative integers such that
\begin{enumerate}
  \item[\emph{(i)}]   $\operatorname{Ext}^{i}_{R}(R/\mathfrak{a}, X)$ is finite for all $i\leq s+t$,
  \item[\emph{(ii)}]  $\operatorname{H}^{i}_{\mathfrak{a}}(M, X)$ is $\mathfrak{a}$--cofinite for all $i< t$, and
  \item[\emph{(iii)}]  $\operatorname{H}^{i}_{\mathfrak{a}}(M, X)$ is minimax for all $t\leq i\leq s+ t$.
\end{enumerate}
Then $\operatorname{H}^{i}_{\mathfrak{a}}(M, X)$ is $\mathfrak{a}$--cofinite for all $i\leq s+ t$.
\end{cor}

\begin{proof}
Follows from Lemma \ref{2-7}, Corollary \ref{2-11}, and \cite[Proposition 4.3]{Mel}.
\end{proof}


\begin{cor}\label{3-5}
Let $M$ be a finite $R$--module, $X$ an arbitrary $R$--module, and $t$ a non-negative integer such that
\begin{enumerate}
\item[\emph{(i)}]  $\operatorname{Ext}^{i}_{R}(R/\mathfrak{a}, X)$ is finite for all $i\leq t$, and
\item[\emph{(ii)}] $\operatorname{H}^{i}_{\mathfrak{a}}(M, X)$ is minimax for all $i< t$.
\end{enumerate}
Then $\operatorname{H}^{i}_{\mathfrak{a}}(M, X)$ is $\mathfrak{a}$--cofinite for all $i< t$ and the $R$--module $\operatorname{Hom}_{R}(R/\mathfrak{a}, \operatorname{H}^{t}_{\mathfrak{a}}(M, X))$ is finite. Moreover, if $\operatorname{Ext}^{t+1}_{R}(R/\mathfrak{a}, X)$ is finite, then $\operatorname{Ext}^{1}_{R}(R/\mathfrak{a}, \operatorname{H}^{t}_{\mathfrak{a}}(M, X))$ is finite.
\end{cor}

\begin{proof}
Since $\operatorname{Hom}_R(R/\mathfrak{a}, \Gamma_{\mathfrak{a}}(M, X))\cong \operatorname{Hom}_R(M/\mathfrak{a}M, X)$ is finite by Lemma \ref{2-7},  $\Gamma_{\mathfrak{a}}(M, X)$ is $\mathfrak{a}$--cofinite from \cite[Proposition 4.3]{Mel}, and so $\operatorname{H}^{i}_{\mathfrak{a}}(M, X)$ is $\mathfrak{a}$--cofinite for all $i< t$ by Corollary \ref{3-4}. Thus the $R$--module $\operatorname{Hom}_{R}(R/\mathfrak{a}, \operatorname{H}^{t}_{\mathfrak{a}}(M, X))$ is finite from Lemma \ref{2-7} and Corollary \ref{2-11}. The last part follows by Lemma \ref{2-7} and Corollary \ref{2-13}.
\end{proof}


Let $X$ be an arbitrary $R$--module such that $\operatorname{Ext}^{i}_{R}(R/\mathfrak{a}, X)\in \mathcal S$ for all $i$ and let $t$ be a non-negative integer. It was shown in \cite[Theorem 4.2]{ATV} that if $\operatorname{H}^{i}_{\mathfrak{a}}(X)$ is $(\mathcal S, \mathfrak{a})$--cofinite for all $i\neq t$, then $\operatorname{H}^{t}_{\mathfrak{a}}(X)$ is $(\mathcal S, \mathfrak{a})$--cofinite (see also \cite[Proposition 2.5]{MV} and \cite[Proposition 3.11]{Mel}). We generalize this result in the following corollary which generalizes and improves \cite[Corollary 3.2.1]{RVA}. Note that, by \cite[Theorem 3.3]{BA}, $\operatorname{Ext}^{i}_{R}(R/\mathfrak{a}, X)$ is finite for all $i$ if and only if $\operatorname{Ext}^{i}_{R}(R/\mathfrak{a}, X)$ is finite for all $i\leq \operatorname{ara}(\mathfrak{a})$.

\begin{cor}\label{3-6}
Let $M$ be a finite $R$--module $($which is not necessarily with finite projective dimension$)$, $X$ an arbitrary $R$--module such that $\operatorname{Ext}^{i}_{R}(R/\mathfrak{a}, X)\in \mathcal S$ for all $i$, and $t$ a non-negative integer such that $\operatorname{H}^{i}_{\mathfrak{a}}(M, X)$ is $(\mathcal S, \mathfrak{a})$--cofinite for all $i\neq t$. Then $\operatorname{H}^{t}_{\mathfrak{a}}(M, X)$ is $(\mathcal S, \mathfrak{a})$--cofinite. In particular, if $\operatorname{Ext}^{i}_{R}(R/\mathfrak{a}, X)$ is finite for all $i\leq \operatorname{ara}(\mathfrak{a})$ and $\operatorname{H}^{i}_{\mathfrak{a}}(M, X)$ is $\mathfrak{a}$--cofinite for all $i\neq t$, then $\operatorname{H}^{t}_{\mathfrak{a}}(M, X)$ is $\mathfrak{a}$--cofinite.
\end{cor}

\begin{proof}
Assume that $s$ is a non-negative integer. Since $\operatorname{Ext}^{i}_{R}(R/\mathfrak{a}, X)\in \mathcal S$ for all $i\leq s+ t$, $\operatorname{Ext}^{s+ t- i}_{R}(\operatorname{Tor}^{R}_{i}(R/\mathfrak{a}, M), X)\in \mathcal S$ for all  $i\leq s+ t$ from Lemma \ref{2-7}. By applying Corollary \ref{2-10} with $N= R/\mathfrak{a}$, it shows that $\operatorname{Ext}^{s}_{R}(R/\mathfrak{a}, \operatorname{H}^{t}_{\mathfrak{a}}(M, X))\in \mathcal S$.
\end{proof}


In the next two results, we generalize and improve \cite[Theorems 3.7 and 3.8]{BNS}.

\begin{cor}\label{3-7}
Let $R$ be a local ring, $\mathfrak{a}$ an ideal of $R$ with $\dim(R/\mathfrak{a})\leq 2$, $M$ a finite $R$--module, $X$ an arbitrary $R$--module, and $t$ a non-negative integer such that $\operatorname{Ext}^i_R(R/\mathfrak{a}, X)$ is finite for all $i\leq t+ 1$. Then the following statements are equivalent:
\begin{enumerate}
\item[\emph{(i)}]  $\operatorname{H}^{i}_{\mathfrak{a}}(M, X)$ is an $\mathfrak{a}$--cofinite $R$--module for all $i< t$;
\item[\emph{(ii)}] $\operatorname{Hom}_R(R/\mathfrak{a}, \operatorname{H}^{i}_{\mathfrak{a}}(M, X))$ is a finite $R$--module for all $i\leq t$.
\end{enumerate}
\end{cor}

\begin{proof}
(i)$\Rightarrow$(ii). This follows from Lemma \ref{2-7} and Corollary \ref{2-11}.

(ii)$\Rightarrow$(i). It follows from Lemma \ref{2-7}, Theorem \ref{2-9}, \cite[Theorem 3.5]{BNS}, and using an induction on $t$.
\end{proof}


\begin{cor}\label{3-8}
Let $R$ be a local ring, $\mathfrak{a}$ an ideal of $R$ with $\dim(R/\mathfrak{a})\leq 2$, $M$ a finite $R$--module, and $X$ an arbitrary $R$--module such that $\operatorname{Ext}^i_R(R/\mathfrak{a}, X)$ is finite for all $i$. Then the following statements hold true.
\begin{enumerate}
\item[\emph{(i)}]  If $\operatorname{H}^{2i}_{\mathfrak{a}}(M, X)$ is $\mathfrak{a}$--cofinite for all $i$, then $\operatorname{H}^{i}_{\mathfrak{a}}(M, X)$ is $\mathfrak{a}$--cofinite for all $i$.
\item[\emph{(ii)}] If $\operatorname{H}^{2i+ 1}_{\mathfrak{a}}(M, X)$ is $\mathfrak{a}$--cofinite for all $i$, then $\operatorname{H}^{i}_{\mathfrak{a}}(M, X)$ is $\mathfrak{a}$--cofinite for all $i$.
\end{enumerate}
\end{cor}

\begin{proof}
Follows from Lemma \ref{2-7}, Theorem \ref{2-9}, \cite[Theorem 3.5]{BNS}, and using an induction on $i$.
\end{proof}


In the following, we state the weakest possible conditions which yield the finiteness of associated prime ideals of generalized local cohomology modules. This result generalizes \cite[Corollary 2.5]{ATV} and improves \cite[Theorem 3.1]{Kh}, \cite[Theorem 3.3]{Mafi}, and \cite[Corollary 3.11]{VA} (see Remark \ref{2-12}). Recall that, an $R$--module $X$ is said to be weakly Laskerian if the set of associated prime ideals of any quotient module of $X$ is finite \cite [Definition 2.1]{DM1}. The category of weakly Laskerian $R$--modules is a Serre subcategory of the category of $R$--modules (see \cite [Lemma 2.3(i)] {DM1}).

\begin{cor}\label{3-9}
Let $X$ be an arbitrary $R$--module, $M$ a finite $R$--module, and $t$ a non-negative integer such that
\begin{enumerate}
  \item[\emph{(i)}]  $\operatorname{Ext}^{t- i}_{R}(\operatorname{Tor}^{R}_{i}(R/\mathfrak{a}, M), X)$ is weakly Laskerian for all $i\leq t$, and
  \item[\emph{(ii)}]   $\operatorname{Ext}^{t+ 1- i}_{R}(R/\mathfrak{a}, \operatorname{H}^{i}_{\mathfrak{a}}(M, X))$ is weakly Laskerian for all $i< t$.
\end{enumerate}
Then $\operatorname{Hom}_{R}(R/\mathfrak{a}, \operatorname{H}^{t}_{\mathfrak{a}}(M, X))$ is weakly Laskerian, and so $\operatorname{Ass}_{R}(\operatorname{H}^{t}_{\mathfrak{a}}(M, X))$ is finite.
\end{cor}

\begin{proof}
Since, by \cite[Exercise 1.2.27]{BH},
$$\operatorname{Ass}_R(\operatorname{Hom}_{R}(R/\mathfrak{a}, \operatorname{H}^{t}_{\mathfrak{a}}(M, X)))= \operatorname{Supp}_R(R/\mathfrak{a})\cap \operatorname{Ass}_{R}(\operatorname{H}^{t}_{\mathfrak{a}}(M, X))= \operatorname{Ass}_{R}(\operatorname{H}^{t}_{\mathfrak{a}}(M, X)),$$
the assertion follows from Corollary \ref{2-11}.
\end{proof}
%
\bibliographystyle{amsplain}

\end{document}